\documentclass[12pt]{amsart}
\usepackage{amsmath}
\usepackage{amssymb}

\theoremstyle{plain}
\newtheorem{theorem}{\bf Theorem}[section]

\newtheorem{lemma}[theorem]{\bf Lemma}

\theoremstyle{definition}

\newtheorem{definition}[theorem]{\bf Definition}

\newcommand{\N}{\mathbb N}
\newcommand{\Z}{\mathbb Z}

 \DeclareMathOperator{\ind}{ind}
 \DeclareMathOperator{\ord}{ord}

 \DeclareMathOperator{\supp}{supp}

\numberwithin{equation}{section}

\begin{document}

\title[]{On the unsplittable minimal zero-sum  sequences over finite cyclic groups of prime order}

\author{Jiangtao Peng}
\address{College of Science, Civil Aviation University of China, Tianjin 300300, P.R. China} \email{jtpeng1982@aliyun.com}

\author{Fang Sun}
\address{College of Science, Civil Aviation University of China, Tianjin 300300, P.R. China} \email{sunfang2005@163.com}

\begin{abstract}
Let $p > 155$ be a prime and let $G$ be a  cyclic group of order $p$. Let $S$ be a minimal zero-sum sequence with elements over $G$, i.e., the sum of elements in $S$ is zero, but no proper nontrivial subsequence of $S$ has sum zero. We call $S$ is unsplittable, if there do not exist $g$ in $S$ and $x,y \in G$ such that $g=x+y$ and $Sg^{-1}xy$ is also a minimal zero-sum sequence. In this paper we show that   if $S$ is an  unsplittable minimal zero-sum sequence of  length  $|S|= \frac{p-1}{2}$, then $S=g^{\frac{p-11}{2}}(\frac{p+3}{2}g)^4(\frac{p-1}{2}g)$ or $g^{\frac{p-7}{2}}(\frac{p+5}{2}g)^2(\frac{p-3}{2}g)$. Furthermore, if $S$ is a minimal zero-sum sequence with $|S| \ge \frac{p-1}{2}$, then $\ind(S) \le 2$.
\end{abstract}

\maketitle

\section{Introduction and Main Results}

Let $G$ be a finite abelian group. The Davenport constant $\mathsf D(G)$ is the smallest integer $\ell \in \N$ such that every sequence $S$ over $G$ of length $|S| \ge \ell$ has a zero-sum subsequence. The studies of the Davenport constant $-$ together with the famous Erd\H{o}s-Ginzburg-Ziv Theorem $-$ is considered as a starting point in zero-sum theory, and it has initiated a huge variety of further research (more information can be found in the surveys \cite{Caro96,GG06,Ge09a}, for recent progress see \cite{GLP14,GG2013,Gryn2013,YZ10}).

The associated inverse problem of Davenport constant studies for the structure of  sequences of length strictly smaller than $\mathsf D(G)$ which do not have a zero-sum subsequence. The index of a sequence is a crucial invariant in the investigation of (minimal) zero-sum sequences (resp. of zero-sum free sequences) over cyclic groups. Recall that the index of a sequence $S$ over $G$ is defined as follows.
\begin{definition}\cite[Definition 5.1.1]{Ge09a}
\begin{itemize}
\item[1.] Let $g \in G$ be a non-zero element with $\ord(g)=n< \infty$. For a sequence  $S=(x_1g)\cdot\ldots\cdot(x_lg)$ over $G$, where $l\in \N_0$ and $1\le x_1, \ldots, x_l \le n$, we define $\|S\|_g=\frac{x_1+\cdots+x_l}{n}$ to be the {\it $g$-norm} of $S$.

\item[2.] Let $S$ be a sequence for which $\langle \supp(S) \rangle \subset G $ is a nontrivial finite cyclic group. Then we call $\ind(S)=\min\{\| S \|_g \,|\,g\in G \, \mbox{with}\,\, \langle \supp(S) \rangle=\langle g \rangle\}$ the {\it index } of $S$.

\item[3.] Let $G$ be a finite cyclic group. $\mathsf{I}(G)$ denotes the smallest integer $l \in \N$ such that every minimal zero-sum sequence $S$ of length $|S| \ge l$ has $\ind(S)=1$.
\end{itemize}
\end{definition}
\noindent Clearly, S has sum zero if and only if ind(S) is an integer. There are also slightly different
definitions of the index in the literature, but they are all equivalent (see Lemma 5.1.2 in~\cite{Ge09a}).

The index of a sequence   was named by Chapman, Freeze and Smith \cite{CFS99}. It was first addressed by Kleitman-Lemke (in the conjecture \cite[page 344]{LK}), used as a key tool by Geroldinger (\cite[page 736]{Ge90}),  and then investigated by Gao \cite{Gao00} in a systematical way. Since then it has received a great deal of attention (see for examples \cite{CS05, GG09, GLPPW11, LP13, LPYZ10,  PL13, Po04, SX13, SXL14,  Xia13, XLP14, XS13,  YZ11}).

To investigated the index of long minimal zero-sum sequences, Gao \cite{Gao00} introduced the invariant $\mathsf{I}(G)$. The precise value of $\mathsf{I}(G)$   has been determined independently by Savchev and  Chen\cite{SC2007}, and by  Yuan\cite{Yuan2007} in 2007.

\begin{theorem}\cite{SC2007,Yuan2007}\label{theorem of sc and yuan}
Let $G$ be a finite cyclic group of order $n$. Then $\mathsf{I}(G)=1$ if $n\in \{1,2,3,4,5,7\}$, $\mathsf{I}(G)=5$ if $n=6$, and $\mathsf{I}(G)=\lfloor \frac{n}{2} \rfloor +2$ if $n \ge 8.$
\end{theorem}


Let $S$ be a minimal zero-sum (resp. zero-sum free) sequence of elements over an abelian group $G$. We say that $S$ is {\it splittable} if there exists an element $g \in \supp(S)$ and two elements $x,y \in  G $ such that $x+y=g$ and $Sg^{-1}xy$ is a minimal zero-sum (resp. zero-sum free) sequence as well; otherwise we say that S is {\it unsplittable}.

Let $S$ be a  minimal zero-sum sequence of length $\mathsf{I}(G)-1$ over a finite cyclic group $G$. If $S$ is splittable, it is easy to check that $\ind(S)=1$. If $S$ is unsplittable, Gao \cite{Gao00} conjectured that $ \ind(S)=2$.  In 2010, Xia and Yuan \cite {XY2010} showed that  Gao's conjecture is true when $n$ is odd, and false when $n$ is even.

\begin{theorem}\cite[Theorem 3.1]{XY2010}\label{theorem of XY2010}
Let $S$ be an unsplittable minimal zero-sum sequence of length $|S|=\mathsf{I}(G)-1$ over a finite cyclic group $G$. We have:
\begin{itemize}
\item[(1)] If $n $ is odd, then $S=g^{\frac{n-5}{2}} (\frac{n+3}{2}g)^{2} (\frac{n-1}{2}g)$ when $n \ge 9$ and $S=g \cdot (3g)^2 \cdot (4g) \cdot (7g)$ when $n=9$. Moreover $\ind(S)=2$.
\item[(2)] If $n$ is even,  then either $S=(2g)^{\frac{n}{2}-1}(x_1g)((n+2-x_1)g),$ where $2 \nmid x_1, 1 < x_1 < n, \ x_1 \neq n+2-x_1$ or $S=g^t (\frac{n}{2}g)((1+\frac{n}{2})g)^{2\ell}$, where $t, \, l$ are positive integers with $t+2\ell =\frac{n}{2}$. Moreover $\ind(S)\ge 2$.
\end{itemize}
\end{theorem}

In this paper, we characterized the unsplittable minimal zero-sum sequences of length $|S|=\mathsf{I}(G)-2$ over a cyclic group $G$ of prime order. Our main results state as following.

\begin{theorem}\label{mainresult1}
Let $p > 155$ be a prime and let $G$ be a cyclic group of order $p$. Let $S$ be an unsplittable minimal zero-sum sequence of length $|S|= \frac{p-1}{2}$ over  $G$. We have $S$ is one of the following forms: \begin{center}$g^{\frac{p-11}{2}}(\frac{p+3}{2}g)^4(\frac{p-1}{2}g)$ or $g^{\frac{p-7}{2}}(\frac{p+5}{2}g)^2(\frac{p-3}{2}g)$.\end{center} Moreover $\ind(S)=2$.
\end{theorem}

\begin{theorem}\label{mainresult2}
Let $p > 155$ be a prime and let $G$ be a cyclic group of order $p$. Let $T$ be a minimal zero-sum sequence of length $|T|\ge \mathsf{I}(G)-2 = \frac{p-1}{2}$ over  $G$. We have $\ind(T)\le 2$.
\end{theorem}

The paper is organized as follows. In the next section, we provide some preliminary results. In section 3, we give a proof for our main results. In the last section, we will give some further remarks.

\section{Preliminaries}

Our notation and terminology are consistent with \cite{GG06} and \cite{Ge09a}. Let $\N_0 = \N \cup \{0 \}$ and for real numbers $a, b$ let $[a,b] = \{x \in \Z \mid a \le x \le b\}$.

Let $G$ be an additive finite abelian group.  Every sequence $S$ over $G$ can be written in the form
\begin{center}
$S=g_1 \cdot  \ldots \cdot g_\ell=\prod_{g \in G} g^{\mathsf v_g (S)}$,\quad with $\mathsf v_g(S) \in \N_0$ for all $g \in G$.
\end{center}
where $\mathsf v_g(S) \in \N_0$ denote the {\it multiplicity} of $g$ in $S$. We call
\begin{itemize}
\item[] $\supp(S) = \{ g \in G \mid \mathsf v_g(S) > 0\}$  the {\it support } of $S$;
\item[] $\mathsf h(S)=\max \{ \mathsf v_g(S) \mid  g\in G\}$  the {\it maximum of the  multiplicities} of $g$ in $S$;
\item[] $|S|= \ell = \sum_{g\in G}\mathsf v_g(S) \in \N_0$ be the {\it length} of $S$;
\item[] $\sigma(S)=\sum_{i=1}^{\ell}g_i = \sum_{g\in G}\mathsf v_g(S)g\in G $ be the {\it sum } of $S$.
\end{itemize}
A sequence $T$ is called a {\it subsequence} of $S$ and denoted by $T \mid S$ if $\mathsf v_g (T) \le \mathsf v_g (S)$ for all $g \in G$. Whenever $T \mid S$, let $ST^{-1}$ denote the subsequence with $T$ deleted from $S$. If $S_1, S_2$ are two disjoint subsequences of $S$, let $$S_1S_2$$ denote the subsequence of $S$ satisfying that $\mathsf v_g(S_1S_2)= \mathsf v_g(S_1) + \mathsf v_g(S_2)$ for all $g \in G$. Let
\begin{center}
$\Sigma(S)=\{ \sigma(T) \mid T \mbox{ is a subsequence of } S  \mbox{ with } 1 \le |T| \le |S|\}.$
\end{center}
The sequence $S$ is called
\begin{itemize}
\item[] {\it zero-sum } \quad if $\sigma(S) = 0 \in G$;
\item[] {\it zero-sum free} \quad if $0 \not\in \Sigma(S)$;
\item[] {\it minimal zero-sum} \quad  if $\sigma(S)=0$ and $\sigma(T)\ne 0$ for every  $T \mid S$  with $1 \le |T| < |S|$.
\end{itemize}

\begin{lemma}\cite[Theorem 5.3.1]{Ge-HK06a}\label{partitionofzerosumfreesequence}
Let $G$ be  an abelian  group. Let $S$ be a zero-sum free sequence over $G$. Suppose $S=S_1S_2\cdots S_t$, then $|\Sigma(S)| \ge \sum_{i=1}^{t}(|\Sigma(S_i)|)$.
\end{lemma}

\begin{lemma}\cite{Bal10}\label{2.2}
Let $p$ be a prime and let $G$ be a cyclic group of order $p$. Suppose $A \subset G$ and $A\cap (-A)=\emptyset$.  Then $|\Sigma(A)|\ge \min \{p, \frac{|A|(|A|+1)}{2}\}.$
\end{lemma}

\begin{lemma}\label{zerosumfreesubset}
Let $p$ be a prime and let $G$ be a cyclic group of order $p$. Let $A$ be a zero-sum free subset of $G$, then $|\Sigma(A)|\ge \min \{p, \frac{|A|(|A|+1)}{2}\}.$
\end{lemma}

\begin{proof}
Since $A$ is a zero-sum free subset, we have $A \cap (-A)=\emptyset$. Hence the results follows from Lemma~ \ref{2.2}.
\end{proof}

\begin{lemma}\cite[Lemma 2.14]{XY2010}\label{sumsetofunsplittable}
Let $p$ be a prime and let $G$ be a cyclic group of order $p$. Suppose $S$ is a minimal zero-sum sequence of elements over $G$. Then $S$ is unsplittable if and only if  $|\Sigma(Sg^{-1})|=p-1$ for every $g \in \supp(S)$.
\end{lemma}

\begin{lemma}\cite[Lemma 2.15]{XY2010}\label{supportoftwo}
Let $p$ be a prime and let $G$ be a cyclic group of order $p$. Let $S$ be a minimal zero-sum sequence consisting of two distinct elements. Then $S $ is splittable.
\end{lemma}

For convenience, from Lemma \ref{coefficient} till Lemma \ref{subsequenceofsupport3} we always assume that $p$ is a prime and $G$ is a cyclic group of order $p$. Let $S$ be an unsplittable minimal zero-sum sequence of elements over $G$.

\begin{lemma}\cite[Lemma 2.5]{XY2010}\label{coefficient}
Suppose $g, tg\in \supp(S)$ with $t \in [2, p-1]$. Then $t \ge \mathsf v_g(S)+2$. Moreover  $t \ne \frac{p+1}{2}$.
\end{lemma}

\begin{lemma}\cite[Lemma 2.6]{XY2010}\label{subsequenceofsupport1}
Suppose $g, h\in \supp(S)$ with $g\ne h$. Then
\begin{itemize}
\item[(1)] If $k\in [0, \mathsf v_g(S)]$, then $|\Sigma(g^kh)| =2k+1$.
\item[(2)] If $\mathsf v_g(S) \ge 2$ and $\mathsf v_h(S) \ge 2$, then $|\Sigma(g^2h^2)|=8$.
\end{itemize}
\end{lemma}

\begin{lemma}\label{subsequenceofsupport2}
Let $T=g^k(xg)^2$ be a subsequence of $S$, where $k \ge 3$. Then  $|\Sigma(T)| \ge 2|T|$. Moreover apart from  the case $T=g^k(\frac{p+3}{2}g)^2$, $|\Sigma(T)| \ge 2|T|+1.$
\end{lemma}

\begin{proof} Since $S$ is unsplittable, by Lemmas \ref{coefficient}, we have $x \ge k+2$ and $x \ne \frac{p+1}{2}.$

If $2x < p$, since $S$ has minimal zero-sum, we have $2x+k < p$. Then $g,2g,\ldots,kg, xg, (x+1)g, \ldots, (x+k)g, 2xg, (2x+1)g, \ldots, (2x+k)g$ are pairwise distinct and hence $|\Sigma(T)| \ge 3k+2 \ge 2|T|+1.$

Next assume that $2x > p$. Then $x \ge \frac{p+3}{2}$. Since $S$ has minimal zero-sum, we have $x+k < p$, and hence $x > 2x-p+k$.

If $2x-p > k$, then $g,2g,\ldots,kg, (2x-p)g, (2x-p+1)g, \ldots, (2x-p+k)g, xg, (x+1)g, \ldots, (x+k)g$ are pairwise distinct and hence $|\Sigma(T)| \ge 3k+2 \ge 2|T|+1.$

If $ 2x-p \le k$, then $g,2g,\ldots,kg, (k+1)g, \ldots,  (2x-p+k)g, xg, (x+1)g, \ldots, (x+k)g$ are pairwise distinct and hence $|\Sigma(T)| \ge (2x-p+k)+(k+1) \ge 2|T|,$ and the equality holds if and only if $x= \frac{p+3}{2}$.
\end{proof}

\begin{lemma}\cite[Lemma 2.11]{XY2010}\label{subsequenceofsupport4}
Let $T=g_1^kg_2g_3$ be a subsequence of $S$. Then $|\Sigma(T)| \ge 2|T|$, moreover apart from  the case $T=g_1^k(\frac{p-1}{2}g_1)(\frac{p+3}{2}g_1)$, $|\Sigma(T)| \ge 2|T|+1.$
\end{lemma}

\begin{lemma}\label{subsequenceofsupport3}
Let $T$ be a subsequence of $S$. If $|\supp(T)| \ge 2$, then there exists $g \in \supp(T)$ such that   $|\Sigma(g^{-1}T)| \ge 2|g^{-1}T|-1$.
\end{lemma}

\begin{proof}
Since $|\supp(T)| \ge 2$, we can write $$T=U_1\cdot \ldots \cdot U_{t}V_1 \cdot \ldots \cdot V_r W,$$ where  $U_1, U_2, \ldots , U_t$ are 3-subsets of $G$, $V_1, V_2, \ldots, V_r$ are of form $g^2h^2$ with $g, h \in \supp(T)$ and $W=g^xh^y$ with $y \le 1$.  By Lemma \ref{zerosumfreesubset} we have  $|\Sigma(U_i)| \ge 6= 2|U_i|$ for $i=1,2,\ldots,t$. By Lemma \ref{subsequenceofsupport1}.2 we have $|\Sigma(V_j)|= 8 = 2|V_j|$ for $j=1,2,\ldots,r$.

If $y=1$, then by Lemma \ref{subsequenceofsupport1}.1 we have $|\Sigma(g^{-1}W)| \ge 2|g^{-1}W|-1$. By Lemma \ref{partitionofzerosumfreesequence}, we infer that $|\Sigma(Tg^{-1})| \ge \sum_{i=1}^{t}|\Sigma(U_i)| + \sum_{j=1}^{r}|\Sigma(V_j)| + |\Sigma(g^{-1}W)| \ge 2\sum_{i=1}^{t}|U_i| + 2\sum_{j=1}^{r}|V_j| + 2|Wg^{-1}|-1 = 2|Tg^{-1}|-1$, and we are done.

If $y=0$, we have that $t \ge 1$ or $r \ge 1$. If $t \ge 1$, in view of Lemmas \ref{zerosumfreesubset} and \ref{subsequenceofsupport4}, there exists $g \in \supp(T)$ such that $|\Sigma(WU_tg^{-1})|\ge 2|WU_tg^{-1}|-1$. Therefore by Lemma \ref{partitionofzerosumfreesequence}, we infer that $|\Sigma(Tg^{-1})| \ge 2|Tg^{-1}|-1$, and we are done. If $r \ge 1$, then by Lemma \ref{subsequenceofsupport1}.1, we have $|\Sigma(WV_rh^{-1})|\ge 2|WV_rh^{-1}|-1$. Also by Lemma \ref{partitionofzerosumfreesequence}, we infer that $|\Sigma(Th^{-1})| \ge 2|Th^{-1}|-1$, and we are done.

This completes the proof.
\end{proof}

\section{Proof of the main results}

Throughout this section, we always assume  that
\begin{itemize}
\item[(1)] $p>155$ is a prime;
\item[(2)] $G$ is a cyclic group of order $p$;
\item[(3)] $S$ is an unsplittable minimal zero-sum sequence of length $\frac{p-1}{2}$ over $G$.
\end{itemize}

\begin{lemma}\label{support}
$3 \le |\supp(S)| \le 4$.
\end{lemma}

\begin{proof}
Since $S$ is unsplittable, by Lemma \ref{supportoftwo}, we have $|\supp(S)| \ge 3$. It remains to show that $|\supp(S)| \le 4$.

Assume to the contrary that $|\supp(S)| \ge 5$. Suppose $S=g_1^{r_1}g_2^{r_2}\cdot \ldots \cdot g_k^{r_k}$, where $r_1 \ge r_2 \ge \cdots \ge r_k \ge 1$ and $k  \ge 5$. We can write $$S=TU,$$ where $T=g_1g_2\cdot \ldots \cdot g_5$ and $|U|=\frac{p-1}{2}-5=\frac{p-11}{2}$. By Lemma \ref{zerosumfreesubset} we have $|\Sigma(T)| \ge 15$.

If $|\supp(U)| \ge 2$, by Lemma \ref{subsequenceofsupport3}, there exists $a \in [1,k]$ such that  $|\Sigma(Ug_a^{-1})| \ge 2|Ug_a^{-1}|-1$. By Lemma \ref{partitionofzerosumfreesequence}, we infer that $|\Sigma(Sg_a^{-1})| \ge  |\Sigma(T)| +  |\Sigma(Ug_a^{-1})| \ge 15  + 2|Ug_a^{-1}|-1 = 15+ 2(\frac{p-11}{2}-1)-1 > p$, yielding a contradiction to Lemma \ref{sumsetofunsplittable}.

Next assume that  $|\supp(U)| = 1$. Then $k=5$ and $U=g_1^{r_1-1}$. Hence we can write $$S=g^{r_1}(t_2g)\cdot \ldots \cdot (t_5g)$$ with $2 \le t_2 < \cdots < t_5 \le p-1$. Then $r_1 =\frac{p-1}{2}-(5-1) = \frac{p-9}{2}$. By Lemma \ref{coefficient}, we have $t_2 \ge r_1+2 = \frac{p-5}{2}$. Since $S$ has minimal zero-sum, we have $t_5 \le p- r_1- 1 = \frac{p+7}{2}$. Since $p \ge 19$ we infer that $r_1+ t_2 +t_3+t_4 +t_5 \not\equiv 0 \pmod p$, yielding a contradiction to that $S$ is zero-sum.

Therefore $|\supp(S)| \le 4$. This completes the proof.
\end{proof}

\begin{lemma}\label{multiplicity3}
Suppose $S=g^{r_1}(t_2g)^{r_2}(t_3g)^{r_3}(t_4g)^{r_4}$, where $ 2 \le  t_2, t_3, t_4 \le p-1$ and $r_2+r_3+r_4 \le 15$. If  $ r_i \ge 2$ for some $i \in \{2,3,4\}$, then $t_i\ge \frac{p+3}{2}$.
\end{lemma}

\begin{proof}
Since $r_2+r_3+r_4 \le 15$, we have $r_1=|S|-r_2-r_3 - r_4\ge \frac{p-31}{2}$. Since $S$ is unsplittable, by Lemma \ref{coefficient} we infer that $t_2, t_3, t_4 \ge r_1+2 \ge \frac{p-27}{2}$ and $t_2, t_3, t_4 \ne
\frac{p+1}{2}$. Since $S$ has minimal zero-sum, $t_2, t_3, t_4 \le p-r_1-1 \le \frac{p+29}{2}$. Then $p-27\le 2t_2, 2t_3, 2t_4\le p+29$.

Next assume that  $r_i \ge 2$ for some $i \in \{2,3,4\}$. If $2t_i < p$, since $S$ is a minimal zero-sum sequence, we infer that $2t_i \le p-r_1 -1 \le \frac{p+29}{2}$, which implies $p \le 83$, yielding a contradiction. Hence $2t_i > p$. Moreover $t_i \ge \frac{p+3}{2}$.
\end{proof}

\begin{lemma}\label{support4}
$\supp(S)=3$.
\end{lemma}

\begin{proof}
By Lemma~\ref{support}, we have $|\supp(S)| \in [3,4]$. Assume that $|\supp(S)|=4$ and $S=g_1^{r_1}g_2^{r_2}g_3^{r_3} g_4^{r_4}$, where $r_1 \ge r_2 \ge r_3 \ge r_4 > 0$.

By Lemma \ref{subsequenceofsupport2}, we have that either $|\Sigma(g_i^3g_j^2)| \ge 11$ or $|\Sigma(g_i^2g_i^3)| \ge 11$ for $i,j \in \{1,2,3\}$. By Lemma \ref{zerosumfreesubset} we have $|\Sigma(g_1g_2g_3 g_4)| \ge 10$.

We first show that $r_4=1$. Assume to the contrary that $r_4 \ge 2$. Write $$S=T_1 \cdot \ldots \cdot T_{r_4}U,$$ where $T_1= \cdots = T_{r_4}=g_1g_2g_3 g_4$, $|\supp(U)| \le 3$.  If $|\supp(U)| \ge 2$, by Lemma \ref{subsequenceofsupport3}, there exist $a \in \{1,2,3\}$ such that $|\Sigma(g_a^{-1}U)| \ge 2 |g_a^{-1}U| -1$. Then by Lemma \ref{partitionofzerosumfreesequence} we infer that $|\Sigma(Sg_a^{-1})| \ge  \sum_{i=1}^{r_4}|\Sigma(T_i)|  + |\Sigma(Ug_a^{-1})| \ge 2(|S|-1)+2r_4-1 \ge p$, yielding a contradiction to Lemma \ref{sumsetofunsplittable}. Hence we may assume that $r_2=r_3=r_4$ and therefore $U=g_1^{r_1-r_4}$. If $r_4 \ge 2$, then by Lemma \ref{subsequenceofsupport4}, there exists $a\in \{2,3,4\}$ such that $|\Sigma(T_{r_4}Ug_a^{-1})| \ge 2|T_{r_4}Ug_a^{-1}|+1$. By Lemma \ref{partitionofzerosumfreesequence}, we infer that $|\Sigma(Sg_a^{-1})| \ge  \sum_{i=1}^{r_4-1}|\Sigma(T_i)| + |\Sigma(T_{r_4}Ug_a^{-1})| \ge 2(|S|-1)+2(r_4-1)+1 \ge p$, yielding a contradiction. Therefore $r_4=1$.

Second we will show that $r_2 \le 7$. Assume to the contrary that  $r_2 \ge 8$. Write $$S=TU_1U_2V,$$ where $T=g_1g_2g_3 g_4$, $U_1=U_2=g_1^3g_2^2$ or $g_1^2g_2^3$ such that $|\Sigma(U_i)| \ge 11=2|U_i|+1$ for $i=1,2$, $|\supp(V)| \ge 2$.  In view of  Lemmas \ref{zerosumfreesubset} and \ref{subsequenceofsupport3}, there exists $a\in \{1,2,3\}$ such that $|\Sigma(g_a^{-1}V)| \ge 2|g_a^{-1}V|-1$. By Lemma \ref{partitionofzerosumfreesequence}, we infer that $|\Sigma(Sg_a^{-1})| \ge  |\Sigma(T)| + \sum_{i=1}^{2}|\Sigma(U_i)|  + |\Sigma(g_a^{-1}V)| \ge 2(|S|-1)+4 -1 \ge p$, yielding a contradiction to Lemma~\ref{sumsetofunsplittable}. Hence $r_2 \le 7.$

Next we will show that $r_3=1$. Assume to the contrary that $r_3 \ge 2$. By Lemma \ref{multiplicity3}, we infer that $t_2, t_3 \ge \frac{p+3}{2}$. Since $p> 155$, we have that $r_1 > 56$.

If $r_3 \ge 4$, write $$S=TU_1U_2V,$$ where $T=g_1g_2g_3 g_4$, $U_1=g_1^3g_2^2$ or $g_1^2g_2^3$, $U_2=g_1^3g_3^2$ or $g_1^2g_3^3$ such that $|\Sigma(U_i)| \ge 11=2|U_i|+1$ for $i=1,2$, $|\supp(V)| \ge 2$.  In view of  Lemmas \ref{zerosumfreesubset} and \ref{subsequenceofsupport3}, there exists $a\in \{1,2,3\}$ such that $|\Sigma(g_a^{-1}V)| \ge 2|g_a^{-1}V|-1$. By Lemma \ref{partitionofzerosumfreesequence}, we infer that $|\Sigma(Sg_a^{-1})| \ge  |\Sigma(T)| + \sum_{i=1}^{2}|\Sigma(U_i)|  + |\Sigma(g_a^{-1}V)| \ge 2(|S|-1)+4 -1 \ge p$, yielding a contradiction to Lemma \ref{sumsetofunsplittable}.

If $r_3 \le 3$, then $r_1 = |S| - r_2 -r_3 -1 \ge \frac{p-23}{2}$. Since $S$ is a minimal zero-sum sequence, we infer that $t_i\le p -r_1 -1 \le \frac{p+21}{2}$ for $i=2,3,4$. Since $ p > 155$, we infer that $g^{r_1}(t_2g)^5(t_3g)^{2}$ contains a zero-sum subsequence. This together with $S$ is a minimal zero-sum sequence forces that $r_2 \le 4$. Hence $r_1 = |S| - r_2 -r_3 -1 \ge \frac{p-17}{2}$. Similarly since  $g^{r_1}(t_2g)^3(t_3g)^{2}$ contains a zero-sum subsequence, we infer that $r_2=r_3=2$. Hence $r_1 = |S| - r_2 -r_3 -1 = \frac{p-11}{2}$. But $g^{r_1}(t_2g)^2(t_3g)^{2}$ contains a zero-sum subsequence, yielding a contradiction to that $S$ is a minimal zero-sum sequence.

Therefore  $r_3=1$.

Since $r_1= |S|- r_2 -r_3 -r_4 = \frac{p-5-2r_2}{2},$ by Lemma \ref{coefficient}, $t_i \ge r_1+2 = \frac{p-1-2r_2}{2}$ for $i=2,3,4$. Since $S$ is a minimal zero-sum sequence, we have $t_i \le p-r_1-1= \frac{p+3+2r_2}{2}$ for $i=2,3,4$. Now assume that $t_i=\frac{p+x_i}{2}$, then $-1-2r_2 \le x_i \le 2r_2+3 $ for $i=2,3,4$.

Since $S$ is a zero-sum sequence, we have
\begin{equation*}\label{equ2}
r_1+t_2r_2+t_3r_3+t_4r_4= \frac{p-5-2r_2}{2}+ \frac{p+x_2}{2}r_2+  \frac{p+x_3}{2} +   \frac{p+x_4}{2}\equiv 0 \pmod p.
\end{equation*}
Since $p$ is odd prime, we have $p-5-2r_2+pr_2+x_2r_2+p+x_3+p+x_4 \equiv 0 \pmod p$. Hence $$(x_2-2)r_2+x_3 +x_4-5 \equiv 0 \pmod p.$$
Recalling that  $S$ is a minimal zero-sum sequence, $p > 155$, $r_2 \le 7$, and  $-1-2r_2 \le x_i \le 2r_2+3 $ for $i=2,3,4$, it is easy to check that
\begin{center}
$r_2=1$, and $\{t_2, t_3, t_4 \}= \{\frac{p-1}{2}, \frac{p+3}{2}, \frac{p+5}{2}\}$.
\end{center}
Therefore $S=g^{\frac{p-7}{2}}(\frac{p-1}{2}g)(\frac{p+3}{2}g)(\frac{p+5}{2}g)$. But $|\Sigma(S(\frac{p+5}{2}g)^{-1})|=p-3$, yielding a contradiction to Lemma~\ref{sumsetofunsplittable}.

Hence $|\supp(S)|=3$. This completes the proof.
\end{proof}

\medskip

\begin{lemma}\label{r_2_r_3}
$S$ is not of form $S=g^{r_1}(\frac{p-1}{2}g)^{r_2}(\frac{p+3}{2}g)^{r_3}$ with $r_1 \ge r_2 \ge r_3 > 0$.
\end{lemma}
\begin{proof}
Assume to the contrary that such $S$ exists. Since $|S|= r_1 + r_2 +r_3 =\frac{p-1}{2}$ and $r_1 \ge r_2 \ge r_3$, we infer that $r_2 \le \frac{p-1}{4}$. Since $S$ is a zero-sum sequence, we have that $r_1 + \frac{p-1}{2}r_2 + \frac{p+3}{2}r_3 \equiv 0 \pmod p.$
Hence $$2r_1 -r_2 + 3r_3 \equiv 0 \pmod p.$$ This together with $r_1+r_2+r_3=\frac{p-1}{2}$ gives that $3r_2-r_3 \equiv p-1 \pmod p$. Which is impossible since $r_2 \ge r_3$ and $r_2 \le \frac{p-1}{4}$.
\end{proof}

\begin{lemma}\label{support3}
Suppose $S=g^{r_1}(t_2g)^{r_2}(t_3g)^{r_3}$, where $r_1 \ge r_2 \ge r_3 > 0$. Then $S=g^{\frac{p-11}{2}}(\frac{p+3}{2}g)^4(\frac{p-1}{2}g)$ or $g^{\frac{p-7}{2}}(\frac{p+5}{2}g)^2(\frac{p-3}{2}g)$.
\end{lemma}
\begin{proof}
By Lemma \ref{sumsetofunsplittable}, if $S=g^{\frac{p-11}{2}}(\frac{p+3}{2}g)^4(\frac{p-1}{2}g)$ or $g^{\frac{p-7}{2}}(\frac{p+5}{2}g)^2(\frac{p-3}{2}g)$, it is easy to check that $S$ is unsplittable. It remains to show that $S$ is of above forms.

Since $S$ is a minimal zero-sum sequence, in view of Lemma \ref{coefficient}, we obtained that $r_2 \ge 2$.

{\bf Case 1.} $t_2 \ne \frac{p+3}{2}$. By Lemma \ref{subsequenceofsupport2}, $|\Sigma(g^3(t_2g)^2)|\ge 11$. By Lemmas \ref{r_2_r_3} and \ref{subsequenceofsupport4}, we infer that $|\Sigma(g^k(t_2g)(t_3g))|\ge 2(k+2)+1$.

Now write $$S= T_1 \cdot \ldots \cdot T_{x} U_1 \cdot \ldots \cdot U_{y} V, $$ where $T_1=\cdots = T_{x}=g^2(t_2g)(t_3g)$, $U_1= \cdots =U_{y} = g^3(t_2g)^2$ and $|\supp(V)| \le 2$. Clearly $x \ge 1$.

If $x+y \ge 5$, then $|\supp(T_xV)| \ge 3$. By Lemma \ref{subsequenceofsupport3}, there exists $a\in \{1,t_2,t_3\}$ such that $|\Sigma((ag)^{-1}T_xV)| \ge 2|(ag)^{-1}T_xV|-1$. By Lemma \ref{partitionofzerosumfreesequence}, we infer that $|\Sigma(S(ag)^{-1})| \ge  \sum_{i=1}^{x-1}|\Sigma(T_i)| +  \sum_{j=1}^{y}|\Sigma(U_j)|  + |\Sigma((ag)^{-1}T_xV)| \ge 2(|S|-1)+(x+y-1) -1\ge p$, yielding a contradiction to Lemma \ref{sumsetofunsplittable}.

If $x+y=4$, since $p >155$, we infer that $|\supp(V)| \ge 1$. If $|\supp(V)|=2$, by Lemma \ref{subsequenceofsupport3}, there exists $a\in \{1,t_2,t_3\}$ such that $|\Sigma((ag)^{-1}V)| \ge 2|(ag)^{-1}V|-1$. By Lemma \ref{partitionofzerosumfreesequence}, we infer that $|\Sigma(S(ag)^{-1})| \ge  \sum_{i=1}^{x}|\Sigma(T_i)| +  \sum_{j=1}^{y}|\Sigma(U_j)|  + |\Sigma((ag)^{-1}V)| \ge 2(|S|-1)+(x+y)-1 \ge p$, yielding a contradiction. If  $|\supp(V)| = 1$,  we infer that $V=g^k$. Then by Lemma \ref{subsequenceofsupport4}, $|\Sigma(g^{-1}T_xV)| \ge 2|g^{-1}T_xV|$. By Lemma \ref{partitionofzerosumfreesequence}, we infer that $|\Sigma(Sg^{-1})| \ge  \sum_{i=1}^{x-1}|\Sigma(T_i)| +  \sum_{j=1}^{y}|\Sigma(U_j)|  + |\Sigma(g^{-1}T_xV)| \ge 2(|S|-1)+(x+y-1) \ge p$, yielding a contradiction.

Therefore $$x+y \le 3.$$ Then $r_3 \le 3$ and moreover $r_2+r_3 \le 7$. Since $r_2 \ge 2$, by Lemma \ref{multiplicity3},  we have $t_2 \ge \frac{p+5}{2}$. Note that $r_1=|S| - r_2 - r_3 \ge \frac{p-15}{2}$.  Since $S$ is a minimal zero-sum sequence, $t_2, t_3 \le p-r_1 -1 \le \frac{p+13}{2}$.  Since  $p >155$, we have $g^{r_1}(t_2g)^{3}$ contains a zero-sum subsequence, yielding a contradiction. Hence we may assume that $r_2 = 2$.

If $r_3 =2$, then $r_1=|S| - r_2 - r_3 =\frac{p-9}{2}$. Since $S$ is a minimal zero-sum sequence, $t_2, t_3 \le p-r_1 -1 = \frac{p+7}{2}$. By Lemma \ref{multiplicity3}, $ t_3\ge \frac{p+3}{2}$.  Since $p >155$, we have $g^{r_1}(t_2g)(t_3g)^2$ contains a zero-sum subsequence, yielding a contradiction. Hence we may assume that $r_3 = 1$. Then $r_1=|S| - r_2 - r_3 = \frac{p-7}{2}$. Since $S$ is a minimal zero-sum sequence, we have $t_2 \le p-r_1-1 = \frac{p+5}{2}$. Therefore $t_2= \frac{p+5}{2}$. Then $t_3= \frac{p-3}{2}$ and hence $S=g^{\frac{p-7}{2}}(\frac{p+5}{2}g)^2(\frac{p-3}{2}g)$.

{\bf Case 2.} $t_2 = \frac{p+3}{2}$. By Lemma \ref{subsequenceofsupport2}, $|\Sigma(g^2(t_2g)^3)|\ge 11$ and $|\Sigma(g^3(t_3g)^2)|\ge 11$.

We first show that $r_2 \le 11$ and $r_3 \le 5$.

If $r_2 \ge 12$, then we can write $$S=T_1T_2T_3T_4U,$$ where $T_1=T_2=T_3=T_4= g^2(t_2g)^3$ and $|\supp(U)|\ge 2$.
By Lemma \ref{subsequenceofsupport3}, there exists $a \in \{1, t_2, t_3\}$ such that $|\Sigma(U(ag)^{-1})| \ge 2|U(ag)^{-1}|-1$. By Lemma \ref{partitionofzerosumfreesequence}, we infer that $|\Sigma(S(ag)^{-1})| \ge  \sum_{i=1}^{4}|\Sigma(T_i)|  + |\Sigma(U(ag)^{-1})| \ge 2(|S|-1)+4 -1\ge p$, yielding a contradiction to Lemma \ref{sumsetofunsplittable}. Hence we may assume that $r_2 \le 11.$ Since $p \ge 100$, then $r_1 \ge 11$. If $r_2 \ge r_3 \ge 6$, then we can write $$S=T_1T_2U_1U_2V,$$ where $T_1=T_2= g^2(t_2g)^3, \, U_1=U_2 = g^3(t_3g)^2 $ and $|\supp(V)|\ge 2$. By Lemma \ref{subsequenceofsupport3}, there exists $a \in \{1, t_2, t_3\}$ such that $|\Sigma(V(ag)^{-1})| \ge 2|V(ag)^{-1}|-1$. Also by Lemma \ref{partitionofzerosumfreesequence}, we infer that $|\Sigma(S(ag)^{-1})|\ge p$, yielding a contradiction. Hence $r_3 \le 5$.

Since $r_1= |S|- r_2 -r_3 = \frac{p-1-2r_2-2r_3}{2},$ by Lemma \ref{coefficient}, $t_3 \ge r_1+2 = \frac{p+3-2r_2-2r_3}{2}$. Since $S$ is a minimal zero-sum sequence, we have $t_3 \le p-r_1-1= \frac{p-1+2r_2+2r_3}{2}$. Now assume that $t_3=\frac{p+x}{2}$, then $3-2r_2-2r_3 \le x \le 2r_2+2r_3-1 $. Hence $-29 \le x \le 31$.

Since $S$ is a zero-sum sequence, we have
\begin{equation*}\label{equ1}
r_1+t_2r_2+t_3r_3= \frac{p-1-2r_2-2r_3}{2}+ \frac{p+3}{2}r_2+  \frac{p+x}{2}r_3 \equiv 0 \pmod p.
\end{equation*}
Since $p$ is an odd prime, we have $p-1-2r_2-2r_3+pr_2+3r_2+pr_3+xr_3 \equiv 0 \pmod p$. Hence $$r_2+(x-2)r_3-1 \equiv 0 \pmod p.$$
Recalling that  $S$ is a minimal zero-sum sequence, $p > 155$, $r_2 \le 11$, $r_3 \le 5$ and $-29 \le x \le 31$, it is easy to check that
\begin{center}
$r_2=4, r_3=1$ and $t_3= \frac{p-1}{2}$.
\end{center}
We are done. This completes the proof.
\end{proof}

\medskip

\begin{lemma}\label{index}
Suppose $g \in G \setminus \{0\}$ and  $S$ if one of the following forms \begin{center} $g^{\frac{p-11}{2}}(\frac{p+3}{2}g)^4(\frac{p-1}{2}g)$ or $g^{\frac{p-7}{2}}(\frac{p+5}{2}g)^2(\frac{p-3}{2}g)$.\end{center} Then $\ind(S)=2$.
\end{lemma}

\begin{proof}
Suppose $h \in G \setminus \{0\}$, then $g=mh$ for some $m\in [1,p-1]$. If $S=g^{\frac{p-11}{2}}(\frac{p+3}{2}g)^4(\frac{p-1}{2}g)$, it is easy to check that $\|S\|_h \ge 2$ and if $g=2h$, then $\|S\|_h = 2$. Hence $\ind(S)=2$. Similarly, we can show that if $S=g^{\frac{p-7}{2}}(\frac{p+5}{2}g)^2(\frac{p-3}{2}g)$, then  $\ind(S)=2$.
\end{proof}

Now we are in a position to proof the main results.

\medskip

\noindent{\it Proof of Theorem \ref{mainresult1}:} Suppose $S$ is an unsplittable minimal zero-sum sequence of length $\frac{p-1}{2}$. By Lemma \ref{support4}, we have $|\supp(S)|=3$. Then by Lemma~ \ref{support3}, $S=g^{\frac{p-11}{2}}(\frac{p+3}{2}g)^4(\frac{p-1}{2}g)$ or $g^{\frac{p-7}{2}}(\frac{p+5}{2}g)^2(\frac{p-3}{2}g)$. By Lemma \ref{index}, we have $\ind(S)=2$. This completes the proof. \qed

\medskip

\noindent{\it Proof of Theorem \ref{mainresult2}:} If $|T| \ge \frac{p+3}{2}$, by Theorem \ref{theorem of sc and yuan}, $\ind(T)=1$.

Next assume that  $|T|= \frac{p+1}{2}$. If  $T$ is unsplittable, by Theorem \ref{theorem of XY2010}.1, we have $\ind(T)=2$. If $T$ is splittable, i.e., there exists $h\in \supp(T)$ and  $x,y \in G$ such that $h=x+y$ and $T'=xyTh^{-1}$ is also a minimal zero-sum sequence of length $\frac{p+3}{2}.$ Then by Theorem~\ref{theorem of sc and yuan}, $\ind(T')=1$. Clearly $\|T\|_g \le \|T'\|_g$ for every $g \in G \setminus \{0\}$. Hence $\ind(T) \le \ind(T')=1$.

If $|T|= \frac{p-1}{2}$, similar to above we can show that $\ind(T) \le 2$. This completes the proof. \qed

\section{Concluding remarks}

Let $p$ be a prime and let $G$ be a cyclic group of order $p$. When $ p < 155$, it is not hard to characterize the structure of  unsplittable minimal zero-sum sequence of length $|S|=\frac{p-1}{2}$. Similar to  Theorems \ref{mainresult1} and \ref{mainresult2}, we can show that

\begin{theorem}\label{4.1}
Let $p > 200$ be a prime and let $G$ be a cyclic group of order $p$. Let $S$ be an unsplittable minimal zero-sum sequence of length $|S|=\frac{p-3}{2}$ over  $G$. We have $S$ is one of the following forms: \begin{center}
$g^{\frac{p-17}{2}}(\frac{p+3}{2}g)^6(\frac{p-1}{2}g)$ or $g^{\frac{p-9}{2}}(\frac{p+7}{2}g)^2(\frac{p-5}{2}g)$.
\end{center}
\end{theorem}

\begin{theorem}\label{4.2}
Let $p > 200$ be a prime and let $G$ be a cyclic group of order $p$. Let $T$ be a minimal zero-sum sequence of length $|T|\ge \mathsf{I}(G)-3 = \frac{p-3}{2}$ over  $G$. We have $\ind(T)\le 2$.
\end{theorem}

\begin{definition}\label{4.3}\quad
\begin{itemize}
\item[1.] Let $n$ be an integer.  $\mathsf{I}(n)$ denotes the maximal value of index of minimal zero-sum sequences $S$ over a cyclic group $G$ of order $n$.

\item[2.] Let $G$ be a finite cyclic group and   $k \ge 1$ be an integer.  $\mathsf{I}_k(G)$ denotes the smallest integer $l \in \N$ such that every minimal zero-sum sequence $S$ of length $|S| \ge l$ has $\ind(S)\le k$.
\end{itemize}
\end{definition}

To determine  $\mathsf{I}(n)$ is proposed by Gao \cite{Gao00}, and he conjectured that $\mathsf{I}(n) \le c\ln n$ for some  absolute constant $c$ \cite[Conjecture 4.2]{Gao00}. If $n \equiv 0 \pmod 8$, let $G$ be a cyclic group of order $n$. Suppose
\begin{center}
$S=g^{\frac{n}{4}} (\frac{n}{2}g)((1+\frac{n}{2})g)^{\frac{n}{4}}$.
\end{center}
Then $\ind(S)=\frac{n}{8}+1$. Hence the conjecture of Gao is not true for $n \equiv 0 \pmod 8$. In fact, the conjecture is also not true for every even $n$ (see Theorem \ref{theorem of XY2010}.2).

Let $G$ be a finite cyclic group of order $n$. Clearly, if $ k \ge \mathsf{I}(n)$, then $\mathsf{I}_k(G)=1$. If $k=1$, then $\mathsf{I}_1(G)=\mathsf{I}(G)$.   By Theorem \ref{4.2}, we infer that $\mathsf{I}_2(G) \le \frac{p-3}{2},$ provided that $n=p$ is prime.

{\bf Problem.} Determine  $\mathsf{I}(n)$ for all integers $n$ and determine $\mathsf{I}_k(G)$  for all the cyclic groups $G$.

\bigskip

\noindent {\bf Acknowledgements}.  This work has been supported by  the National Science Foundation of China (Grant Nos. 11271207 and 11301531) and a research grant from  Civil Aviation University of China (No. 2010QD02X).

\end{document}